\documentclass[10pt, reqno]{amsart}

\def\BibTeX{{\rm B\kern-.05em{\sc i\kern-.025em b}\kern-.08em
    T\kern-.1667em\lower.7ex\hbox{E}\kern-.125emX}}

\newtheorem{thm}{Theorem}[section]

\newtheorem{lem}[thm]{Lemma}

\theoremstyle{definition}

\theoremstyle{remark}

\numberwithin{equation}{section}

    \newcommand{\floor}[1]{\lfloor#1\rfloor}

    \newcommand{\EE}{\mathbb{E}}

    \renewcommand{\Pr}{\operatorname{P}}

    \newcommand{\dto}{\xrightarrow{d}}
    
    \newcommand{\vto}{\xrightarrow{v}}

    \newcommand{\rmd}{\mathrm{d}}

\newcommand{\PP}{\mathrm{P}}

\newcommand{\be}{\begin{equation}}
    \newcommand{\ee}{\end{equation}}

\begin{document}

\title[A note on joint convergence of partial sum and maxima for linear processes] 
{A note on joint functional convergence of partial sum and maxima for linear processes}

%
\author{Danijel Krizmani\'{c}}

\address{Danijel Krizmani\'{c}\\ Department of Mathematics\\
        University of Rijeka\\
        Radmile Matej\v{c}i\'{c} 2, 51000 Rijeka\\
        Croatia}
\email{dkrizmanic@math.uniri.hr}





\begin{abstract}
Recently, for the joint partial sum and partial maxima processes constructed from linear processes with independent identically distributed innovations that are regularly varying with tail index $\alpha \in (0, 2)$, a functional limit theorem with the Skorohod weak $M_{2}$ topology has been obtained. In this paper we show that, if all the coefficients of the linear processes are of the same sign, the functional convergence holds in the stronger topology, i.e. in the Skorohod weak $M_{1}$ topology on the space of $\mathbb{R}^{2}$--valued c\`{a}dl\`{a}g functions on $[0, 1]$.
\end{abstract}

\maketitle

\section{Introduction}
\label{intro}

Let $(Z_{i})_{i \in \mathbb{Z}}$ be a sequence of i.i.d. random variables. A linear process with innovations $(Z_{i})$ is
a stochastic process of the form
$$ X_{i} = \sum_{j=-\infty}^{\infty}\varphi_{j}Z_{i-j}, \qquad i \in \mathbb{Z},$$
where the constants $\varphi_{j}$ are such that the above series is a.s. convergent. One sufficient condition for that, in the case when $Z_{i}$ is regularly varying with index of regular variation $\alpha >0$, is
\begin{equation}\label{e:asconv}
 \sum_{j=-\infty}^{\infty}|\varphi_{j}|^{\delta} < \infty \quad \textrm{for some} \ 0 < \delta < \alpha,\,\delta \leq 1,
 \end{equation}
 see Theorem 2.1 in Cline~\cite{Cl83} or Resnick~\cite{Re87}, Section 4.5.
 
The literature is very rich with applications of linear processes
in statistical analysis and time series modeling. We refer to Brockwell and Davis~\cite{BD96} for an introduction to the topic.

In this paper we deal with linear processes with heavy-tailed innovations, i.e. we assume $Z_{i}$ is regularly varying with index of regular variation $\alpha \in (0,2)$. In particular, this means that
$$ \Pr(|Z_{i}| > x) = x^{-\alpha} L(x), \qquad x>0,$$
where $L$ is a slowly varying function at $\infty$. Regular
variation of $Z_{i}$ can be expressed in terms of
vague convergence of measures on $\EE = \overline{\mathbb{R}} \setminus \{0\}$:
\begin{equation}
  \label{eq:onedimregvar}
  n \Pr( a_n^{-1} Z_i \in \cdot \, ) \vto \mu( \, \cdot \,) \qquad  \textrm{as} \ n \to \infty,
\end{equation}
where $(a_{n})$ is a sequence of positive real numbers such that
\be\label{eq:niz}
n \Pr (|Z_{1}|>a_{n}) \to 1,
\ee
as $n \to \infty$, and the measure $\mu$ on $\EE$ is given by
\begin{equation}
\label{eq:mu}
  \mu(\rmd x) = \bigl( p \, 1_{(0, \infty)}(x) + r \, 1_{(-\infty, 0)}(x) \bigr) \, \alpha |x|^{-\alpha-1} \, \rmd x,
\end{equation}
where
\be\label{eq:pq}
p =   \lim_{x \to \infty} \frac{\Pr(Z_i > x)}{\Pr(|Z_i| > x)} \qquad \textrm{and} \qquad
  r =   \lim_{x \to \infty} \frac{\Pr(Z_i < -x)}{\Pr(|Z_i| > x)}.
\ee
When $\alpha \in (1,2)$ it holds that $\mathrm{E}(Z_{1}) < \infty$.



Suppose the coefficients $\varphi_{j}$ are all of the same sign, and that $Z_{1}$ is symmetric when $\alpha =1$. Assume also $p>0$ if the coefficients $\varphi_{j}$ are non-negative, and $r>0$ if these coefficients are non-positive. Put $ \beta = \sum_{i=-\infty}^{\infty}\varphi_{i}$ and $\gamma = \max\{|\varphi_{i}| : i \in \mathbb{Z}\}>0$. Condition (\ref{e:asconv}) implies $\beta$ is finite. By Theorem 4.1 in Krizmani\'{c}~\cite{Kr17-2} we have, as $n \to \infty$,
\begin{equation}\label{e:convM2}
L_{n}(\,\cdot\,) := (V_{n}(\,\cdot\,), W_{n}(\,\cdot\,)) \dto (\beta V(\,\cdot\,), \gamma W(\,\cdot\,))
\end{equation}
in $D([0,1], \mathbb{R}^{2})$ endowed with the weak $M_{2}$ topology, where
\be\label{eq:defVn}
V_{n}(t) = \frac{1}{a_{n}} \Bigg( \sum_{i=1}^{\floor {nt}}X_{i} - \floor {nt}b_{n}\Bigg), \qquad W_{n}(t)= \frac{1}{a_{n}} \bigvee_{i=1}^{\floor {nt}}X_{i}, \qquad t \in [0,1],
\ee
 with $(a_n)$ as in~\eqref{eq:niz} and
$$ b_{n} = \left\{ \begin{array}{cc}
                                   0, & \quad \alpha \in (0,1]\\
                                   \beta \mathrm{E}(Z_{1}), & \quad \alpha \in (1,2)
                                 \end{array}\right.,$$
 $V(\,\cdot\,)$ is an $\alpha$--stable L\'{e}vy process and $W(\,\cdot\,)$ is an extremal process.
The purpose of this paper is to strength the convergence in (\ref{e:convM2}) to convergence with respect to the stronger $M_{1}$ topology. We will use the famous compactness approach, i.e. the "finite-dimensional convergence plus tightness" procedure, developed in detail in Billingsley~\cite{Bi68}.

The paper is organized as follows. In Section~\ref{S:M2top} we recall Skorohod's $M_{1}$and $M_{2}$ topologies. Section~\ref{S:Weakconv} is devoted to a short description of weak convergence theory for the space $D([0,1], \mathbb{R}^{2})$, and in Section~\ref{S:Mainresult} we state and prove the our main result.

\section{Skorohod topologies}\label{S:M2top}

We start with a definition of the Skorohod $M_{1}$ topology in the univariate case, i.e. on the space $D([0,1], \mathbb{R})$ of c\`{a}dl\`{a}g functions from $[0,1]$ to $\mathbb{R}$.
For $x \in D([0,1],
\mathbb{R})$ the completed graph of $x$ is the set
\[
  \Gamma_{x}
  = \{ (t,z) \in [0,1] \times \mathbb{R}: z \in [x(t-), x(t)]\},
\]
where $x(t-)$ is the left limit of $x$ at $t$. We define an
order on the graph $\Gamma_{x}$ by saying that $(t_{1},z_{1}) \le
(t_{2},z_{2})$ if either (i) $t_{1} < t_{2}$ or (ii) $t_{1} = t_{2}$
and $|x(t_{1}-) - z_{1}| \le |x(t_{2}-) - z_{2}|$. A parametric representation
of the graph $\Gamma_{x}$ is a continuous nondecreasing function $(r,u)$
mapping $[0,1]$ onto $\Gamma_{x}$, with $r$ being the
time component and $u$ being the spatial component. Let $\Pi(x)$ denote the set of all
parametric representations of the graph $\Gamma_{x}$. For $x_{1},x_{2}
\in D([0,1], \mathbb{R})$ define
\[
  d_{M_{1}}(x_{1},x_{2})
  = \inf \{ \|r_{1}-r_{2}\|_{[0,1]} \vee \|u_{1}-u_{2}\|_{[0,1]} : (r_{i},u_{i}) \in \Pi(x_{i}), i=1,2 \},
\]
where $\|x\|_{[0,1]} = \sup \{ \|x(t)\| : t \in [0,1] \}$. This definition introduces $d_{M_{1}}$ as a metric
on $D([0,1], \mathbb{R})$, and the induced topology is called Skorohod $M_{1}$ topology. It is weaker
than the more frequently used Skorohod $J_{1}$ topology.


If we replace $\Gamma_{x}$ above with
$$G_{x} = \{ (t,z) \in [0,1] \times \mathbb{R}^{2} : z \in [[x(t-), x(t)]]\},$$
where  $[[a,b]]$ is the product segment, i.e.
$[[a,b]]=[a_{1},b_{1}] \times [a_{2},b_{2}]$
for $a=(a_{1}, a_{2}), b=(b_{1}, b_{2}) \in
\mathbb{R}^{2}$, and as parametric representations
of the graph $G_{x}$ we take continuous nondecreasing functions $(r,u)$
mapping $[0,1]$ into $G_{x}$ such that $r(0)=0,
r(1)=1$ and $u(1)=x(1)$, then we obtain the so-called weak $M_{1}$ topology on $D([0,1], \mathbb{R}^{2})$. This topology is weaker than the standard $M_{1}$ topology, but it coincides with the product topology, which is the appropriate topology on $D([0,1], \mathbb{R}^{2})$ for our considerations. The product topology is induced by the metric
\begin{equation}\label{e:defdp}
 d_{p}(x_{1},x_{2})= \max \{ d_{M_{1}}(x_{1j},x_{2j}) : j=1,2 \}
\end{equation}
 for $x_{i}=(x_{i1}, x_{i2}) \in D([0,1],
 \mathbb{R}^{2})$ and $i=1,2$.
For detailed discussion of the $M_{1}$ topologies we refer to
Whitt~\cite{Whitt02}, sections 12.3--12.5.

Recall here also the Skorohod weak $M_{2}$ topology on $D([0,1], \mathbb{R}^{2})$. It is induced by the metric
\begin{equation}\label{e:defdpM2}
 d_{p,M_{2}}(x_{1},x_{2})= \max_{j=1,2}d_{M_{2}}(x_{1j},x_{2j})
\end{equation}
 for $x_{i}=(x_{i1}, x_{i2}) \in D([0,1],
 \mathbb{R}^{2})$, $i=1,2$, where
 $$ d_{M_{2}}(y_{1}, y_{2}) = \bigg(\sup_{a \in \Gamma_{y_{1}}} \inf_{b \in \Gamma_{y_{2}}} d(a,b) \bigg) \vee \bigg(\sup_{a \in \Gamma_{y_{2}}} \inf_{b \in \Gamma_{y_{1}}} d(a,b) \bigg), \quad y_{1},y_{2} \in D([0,1], \mathbb{R}),$$
is the Hausdorff metric on the spaces of graphs, and $d$ is the metric on $\mathbb{R}^{2}$ defined by $d(a,b)=|a_{1}-b_{1}| \vee |a_{2}-b_{2}|$ for $a=(a_{1},a_{2}), b=(b_{1},b_{2}) \in \mathbb{R}^{2}$ (see
Whitt~\cite{Whitt02}, sections 12.10--12.11).

\section{Weak convergence in $D([0,1], \mathbb{R}^{d})$ for $d=1$ and $d=2$}
\label{S:Weakconv}

It is well known that the space $D([0,1], \mathbb{R})$ equipped with the Skorohod $J_{1}$ topology is a Polish space (i.e. metrizable as a complete separable metric space), see for example Billingsley~\cite{Bi68}, Chapter 3. The same holds for the $M_{1}$ topology, since it is topologically complete (see Whitt~\cite{Whitt02}, Section 12.8) and separability remains preserved in the weaker topology.

The standard procedure of proving weak convergence of stochastic processes is to prove convergence of finite-dimensional distributions and relative compactness. For Polish spaces, by Prohorov theorem (see Prohorov~\cite{Pr56}) tightness is necessary and sufficient for relative compactness.

Since our stochastic processes have discontinuities, we will require convergence of the finite-dimensional distributions only for time points that are a.s. continuity points of the limit. For $x \in D([0,1], \mathbb{R}^{d})$ let $\textrm{Disc}(x)$ be the set of discontinuity points of $x$. For a stochastic process $Y$ let
$$ T_{Y} = \{ t \in (0,1] : \Pr (t \in \textrm{Disc}(Y))=0 \} \cup \{1 \}.$$

Now we state the criteria for convergence in distribution in $D([0,1], \mathbb{R})$ equipped with Skorohod $M_{1}$ topology based on Theorem 11.6.6 in Whitt~\cite{Whitt02} (see also theorems 3.2.1 and 3.2.2 in Skorohod~\cite{Sk56}). Let $(Y_{n})$ be a sequence of random elements of $D([0,1], \mathbb{R})$.

\begin{thm}\label{t:Whitt}
There is convergence in distribution $Y_{n} \dto Y$ in $D([0,1], \mathbb{R})$ with the $M_{1}$ topology if and only if
\begin{itemize}
\item[(1)] $ (Y_{n}(t_{1}), \ldots, Y_{n}(t_{k}) \to (Y(t_{1}), \ldots, Y(t_{k}))$ in $\mathbb{R}^{k}$,
for all positive integers $k$ and $t_{1}, \ldots, t_{k} \in T_{Y}$ such that $0 \leq t_{1} < t_{2} < \ldots < t_{k} \leq 1$,
\item[(2)] the sequence $(X_{n})$ is tight (with respect to the $M_{1}$ topology).
\end{itemize}
\end{thm}
Necessity in the above theorem follows from the fact that the space is Polish.
Now we turn our attention to the space $D([0,1], \mathbb{R}^{2})$ endowed with the weak $M_{1}$ topology. Since this topology coincides with the product topology induced by the metric $d_{p}$ in (\ref{e:defdp}), repeating the arguments from Ferger and Vogel in~\cite{FeVo10} where they developed a convergence theory for the Skorohod product space with the $J_{1}$ topology, the following result follows.

\begin{thm}\label{t:M1weak}
Let $(Y_{n}, Q_{n})$ be a sequence of random elements in $D([0,1], \mathbb{R}^{2})$. If
\begin{itemize}
\item[(1)] the sequences $(Y_{n})$ and $(Q_{n})$ are tight with respect to the $M_{1}$ topology,
\item[(2)] there is a random element $(Y,Q)$ in $D([0,1], \mathbb{R}^{2})$ such that, as $n \to \infty$,
\begin{equation}\label{e:findimconv}
   (Y_{n}(t_{1}), \ldots, Y_{n}(t_{k}), Q_{n}(t_{1}), \ldots, Q_{n}(t_{k})) \to (Y(t_{1}), \ldots, Y(t_{k}), Q(t_{1}), \ldots, Q(t_{k}))
\end{equation}
  for all $k \in \mathbb{N}$ and $t_{1}, \ldots, t_{k} \in T_{Y} \cap T_{Q}$,
\end{itemize}
then $(Y_{n}, Q_{n}) \dto (Y,Q)$ with respect to the weak $M_{1}$ topology.
\end{thm}

In~\cite{FeVo10} this result was proven for the $J_{1}$ topology using five lemmas. Three of these lemmas hold trivially in our case, but two of them (Lemma 5.4 and Lemma 5.7) have to be checked with respect to the $M_{1}$ topology, and this is accomplished in the following two lemmas.
Let $T=\{t_{1}, \ldots, t_{k}\}$ and $S=\{s_{1}, \ldots, s_{l}\}$ where $0 \leq t_{1} < t_{2} < \ldots < t_{k} \leq 1$ and $0 \leq s_{1} < s_{2} < \ldots < s_{l} \leq 1$. Define the projections $\pi_{T}$ from $D([0,1], \mathbb{R})$ to $\mathbb{R}^{k}$ and $\pi_{T, S}$ from $D([0,1], \mathbb{R}^{2})$ to $\mathbb{R}^{k+l}$ by
$$ \pi_{T}(x) = (x(t_{1}), \ldots, x(t_{k})), \qquad \pi_{T, S}(x,y) = (x(t_{1}), \ldots, x(t_{k}), y(s_{1}), \ldots, y(s_{l})).$$

\begin{lem}\label{l:contprojection}
If $T \subseteq T_{Y} \cap T_{Q}$, $T$ finite, then $\pi_{T, T}$ is a.s. continuous with respect to the distribution of $(Y,Q)$.
\end{lem}
\begin{proof}
Since $M_{1}$ convergence implies local uniform convergence at all continuity points (see Lemma 12.5.1 in Whitt~\cite{Whitt02}), it follows that $\pi_{T,T}$ is continuous at $z=(x,y) \in D([0,1], \mathbb{R}^{2})$ if $T \subseteq [\textrm{Disc}(z)]^{c}$. Hence $z \in \textrm{Disc}(\pi_{T,T})$ implies the existence of $t \in T$ such that $t \in \textrm{Disc}(z)$. Since $T \subseteq T_{Y} \cap T_{Q}$ we have
$$ \Pr [(Y,Q) \in \textrm{Disc}(\pi_{T,T}) ] \leq \sum_{t \in T} \Pr[t \in \textrm{Disc}((Y,Q))]  =0.$$
\end{proof}

For any $T_{0} \subseteq [0,1]$ let
$$ \mathcal{F}(T_{0}) = \{ \pi_{T}^{-1}(A) : A \in \mathcal{B}(\mathbb{R}^{|T|}), T \subseteq T_{0}, |T| < \infty \},$$
where $\mathcal{B}(\mathbb{R}^{k})$ is the class of Borel sets in $\mathbb{R}^{k}$ and $|T|$ is the cardinal number of $T$.

\begin{lem}
If $T_{0}$ is dense in $[0,1]$ and contains $1$, then $\mathcal{F}(T_{0})$ generates $\mathcal{D}_{M_{1}}$, the $\sigma$--field of Borel sets for the $M_{1}$ topology.
\end{lem}
\begin{proof}
Since the $M_{1}$ topology is weaker than Skorohod $J_{1}$ topology (see Theorem 12.3.2 in Whitt~\cite{Whitt02}), it holds that $\mathcal{D}_{M_{1}} \subseteq \mathcal{D}_{J_{1}}$. By Theorem 14.5 in Billingsley~\cite{Bi68} $\mathcal{F}(T_{0})$ generates $\mathcal{D}_{J_{1}}$. Using the fact that $M_{1}$ convergence implies local uniform convergence at all continuity points (see Lemma 12.5.1 in Whitt~\cite{Whitt02}), similar to the procedure in Billingsley~\cite{Bi68} for the $J_{1}$ topology, we obtain  that $\pi_{T}$ is measurable with respect to $\mathcal{D}_{M_{1}}$, and hence $\sigma( \mathcal{F}(T_{0})) \subseteq \mathcal{D}_{M_{1}}$. Finally we have
$$\mathcal{D}_{J_{1}} = \sigma (\mathcal{F}(T_{0})) \subseteq \mathcal{D}_{M_{1}} \subseteq \mathcal{D}_{J_{1}},$$
i.e. $\sigma (\mathcal{F}(T_{0})) = \mathcal{D}_{M_{1}}$.
\end{proof}

In the proof of our main result in the next section we will need the following result, which state that weak convergence with respect to the weak $M_{2}$ topology implies convergence of finite-dimensional distributions.

\begin{lem}\label{l:M2conv}
If $(Y_{n}, Q_{n}) \dto (Y,Q)$ in $D([0,1], \mathbb{R}^{d})$ equipped with the weak $M_{2}$ topology, then (\ref{e:findimconv}) holds.
\end{lem}
\begin{proof}
Take arbitrary $k \in \mathbb{N}$ and $T = \{t_{1}, \ldots, t_{k} \} \subseteq T_{Y} \cap T_{Q}$.
Similar as in the proof of Lemma~\ref{l:contprojection} we obtain
$$ \Pr [(Y,Q) \in \textrm{Disc}(\pi_{T,T}) ] =0.$$
An application of the continuous mapping theorem yields
$$ \pi_{T,T}(Y_{n},Q_{n}) \to \pi_{T,T}(Y,Q),$$
and hence (\ref{e:findimconv}) holds.
\end{proof}

\section{Main result}
\label{S:Mainresult}

Let $(X_{i})$ be a sequence of linear processes
$$ X_{i} = \sum_{j=-\infty}^{\infty}\varphi_{j}Z_{i-j}, \qquad i \in \mathbb{Z},$$
with regularly varying innovations $Z_{i}$ with index $\alpha \in (0,2)$, and coefficients $\varphi_{j}$ satisfying (\ref{e:asconv}). The theorem below shows that the joint partial sum and maxima
processes $L_{n}(\,\cdot\,)$ from (\ref{e:convM2}) satisfy a functional limit theorem in the weak $M_{1}$ topology with the limit consisting of an $\alpha$--stable L\'{e}vy process and an extremal process.
Put $ \beta = \sum_{i=-\infty}^{\infty}\varphi_{i}$ and $\gamma = \max\{|\varphi_{i}| : i \in \mathbb{Z}\}$.

Recall some basic facts on L\'{e}vy processes and extremal processes. The distribution of a L\'{e}vy process $V (\,\cdot\,)$ is characterized by its
characteristic triple, i.e. the triple
$(a, \nu', b)$ such that
 $$
  \mathrm{E} [e^{izV(1)}] = \exp \biggl( -\frac{1}{2}az^{2} + ibz + \int_{\mathbb{R}} \bigl( e^{izx}-1-izx 1_{[-1,1]}(x) \bigr)\,\nu'(\rmd x) \biggr)$$
for $z \in \mathbb{R}$, where $a \ge 0$, $b \in \mathbb{R}$ are constants, and $\nu'$ is a measure on $\mathbb{R}$ satisfying
$$ \nu' ( \{0\})=0 \qquad \text{and} \qquad \int_{\mathbb{R}}(|x|^{2} \wedge 1)\,\nu'(\rmd x) < \infty.$$
 We refer to Sato~\cite{Sa99} for a textbook treatment of
L\'{e}vy processes.
The distribution of a nonnegative extremal process $W(\,\cdot\,)$ is characterized by its exponent measure $\nu''$ in the following way:
$$ \PP (W(t) \leq x ) = e^{-t \nu''(x,\infty)}$$
for $t>0$ and $x>0$, where $\nu''$ is a measure on $(0,\infty)$ satisfying
$ \nu'' (\delta, \infty) < \infty$
for any $\delta >0$ (see Resnick~\cite{Resnick07}, page 161).

\begin{thm}\label{t:InfMA}
Let $(Z_{i})_{i \in \mathbb{Z}}$ be an i.i.d. sequence of regularly varying random variables with index $\alpha \in (0,2)$. When $\alpha=1$, suppose further that $Z_{1}$ is symmetric. Let $(\varphi_{i})_{i \in \mathbb{Z}}$ be a sequence of real numbers satisfying (\ref{e:asconv}) and assume all of them are of the same sign. If the coefficients $\varphi_{j}$ are non-negative assume also $p>0$, and if they are non-positive assume $r>0$, with $p$ and $r$ as given in (\ref{eq:pq}).
Then, as $n \to \infty$,
$$ L_{n}(\,\cdot\,) := (V_{n}(\,\cdot\,), W_{n}(\,\cdot\,)) \dto (\beta V(\,\cdot\,), \gamma W(\,\cdot\,))$$
in $D([0,1], \mathbb{R}^{2})$ endowed with the weak $M_{1}$ topology, where $V$ is an $\alpha$--stable L\'{e}vy process with characteristic triple $(0,\mu,b)$, with $\mu$ as in $(\ref{eq:mu})$ and
$$ b = \left\{ \begin{array}{cc}
                                   0, & \quad \alpha = 1\\[0.4em]
                                   (p-r)\frac{\alpha}{1-\alpha}, & \quad \alpha \in (0,1) \cup (1,2)
                                 \end{array}\right.,$$
and $W$ is an extremal process with exponent measure
$$\nu(dx)= c\alpha x^{-\alpha-1}1_{(0,\infty)}(x)\,dx,$$
where
$$ c = \left\{ \begin{array}{cc}
                                   p, & \quad \min \{\varphi_{j} : j=0,\ldots, q\} \geq 0\\[0.4em]
                                   r, & \quad \max \{\varphi_{j} : j=0,\ldots,q\} \leq 0
                                 \end{array}\right..$$
\end{thm}
\begin{proof}
By Theorem 4.1 in Krizmani\'{c}~\cite{Kr17-2}, $L_{n}$ converges in distribution to $(\beta V, \gamma W)$ in $D([0,1], \mathbb{R}^{2})$ equipped with the weak $M_{2}$ topology. Then by Lemma~\ref{l:M2conv} we obtain finite-dimensional convergence of $L_{n}$ toward $(\beta V, \gamma W)$.

It is well known that $V_{n} \dto \beta V$ in $D([0,1], \mathbb{R})$ equipped with the $M_{1}$ topology, see for instance Corollary 1 in Tyran-Kami\'{n}ska~\cite{Ty10b}. By Theorem~\ref{t:Whitt} the sequence $(V_{n})$ is tight.

As for the process $W_{n}$, we first approximate it by a sequence of finite order linear processes as in Krizmani\'{c}~\cite{Kr17-2}. For these processes, by Proposition 4.1 in Basrak and Tafro~\cite{BaTa16} we obtain convergence with respect to the $M_{1}$ topology, and then we show that the error of approximation is negligible in the limit (see the proof of Theorem 4.1 in Krizmanic~\cite{Kr17-2} for details). Hence $W_{n} \dto \gamma W$ in $D([0,1], \mathbb{R})$ equipped with the $M_{1}$ topology, and hence by Theorem~\ref{t:Whitt} the sequence $(W_{n})$ is also tight.
An application of Theorem~\ref{t:M1weak} concludes the proof.
\end{proof}

\section*{Acknowledgements}
This work has been supported in part by Croatian Science Foundation under the project 3526 and by University of Rijeka research grant 13.14.1.2.02.


\end{document}